\documentclass[11pt,reqno]{amsart}

\usepackage{amsthm}
\usepackage{amscd}
\usepackage{amsfonts}
\usepackage{amssymb}
\usepackage{amsgen}
\usepackage{amsmath}
\usepackage{amsopn}
\usepackage{verbatim}
\usepackage{xypic}
\usepackage{xspace}
\usepackage{multicol}
\usepackage{url}
\usepackage{upref}

\theoremstyle{plain}
\newtheorem{thm}{Theorem}[section]
\newtheorem{lem}[thm]{Lemma}
\newtheorem{prop}[thm]{Proposition}

\newtheorem{conj}[thm]{Conjecture}

\theoremstyle{definition}

\theoremstyle{remark}
\newtheorem{rem}[thm]{Remark}

 \font\cyr=wncyr10
 \newcommand{\nc}{\newcommand}

\nc{\per}[1]{\underset{#1}{\boldsymbol \pi}\,}

 \nc{\MT}{{\rm MT}}
  \nc{\bgz}{{\bar{\gz}}}
 \nc{\wt}{{\rm wt}}
 \nc{\wht}{{\widehat}}
 \nc{\bwg}{{\bigwedge}}
 \nc{\mmu}{{\boldsymbol{\mu}}}
 \nc{\mal}{{{\scriptstyle \maltese}}}
 \nc{\fA}{{\mathfrak A}}
 \nc{\HH}{{\mathfrak H}}
 \nc{\ra}{\rightarrow}
 \nc{\ors}{{\vec s\,}}
 \nc{\os}{{\overset}}
 \nc{\G}{{\mathbb G}}
 \nc{\Z}{{\mathbb Z}}
 \nc{\R}{{\mathbb R}}
 \nc{\N}{{\mathbb N}}
 \nc{\ZN}{{\mathbb Z_{\ge 0}}}
 \nc{\Q}{{\mathbb Q}}
 \nc{\C}{{\mathbb C}}
 \nc{\Cnn}{{\mathbb C}_{\ge 0}}
 \nc{\Cp}{{\mathbb C}_{>0}}
 \nc{\MPV}{{\mathcal{MPV}}}
 
 \nc{\tB}{{\tilde B}}
 \nc{\tI}{{\tilde I}}
 \nc{\tJ}{{\tilde J}}
 \nc{\tK}{{\tilde K}}
 \nc{\Li}{{\rm Li}}
 \nc{\suf}{{\ast\,}}
 \nc{\sufq}{{\ast_q\,}}
 \nc{\gam}{{\gamma}}
 \nc{\gG}{{\Gamma}}
 \nc{\om}{{\omega}}
 \nc{\vep}{{\varepsilon}}
 \nc{\ga}{{\alpha}}
 \nc{\gl}{{\lambda}}
 \nc{\gb}{{\beta}}
 \nc{\gd}{{\delta}}
 \nc{\gs}{{\sigma}}
 \nc{\gS}{{\Sigma}}

 \nc{\gk}{{\kappa}}
 \nc{\tgz}{{\tilde{\zeta}}}
 \nc{\gO}{{\Omega}}
 \nc{\sif}{{\mathcal S}}
 \nc{\gt}{{\tau}}
\nc{\gz}{{\zeta}}
 \nc{\Lra}{\Longrightarrow}
 \nc{\lra}{\longrightarrow}
 \nc{\fS}{{\mathfrak S}}
 \nc{\DD}{{\mathfrak D}}

 \nc{\Llra}{\Longleftrightarrow}
 \nc{\ol}{\overline}
 \nc{\lms}{\longmapsto}
 \nc{\cv}{{{\mathsf c}{\mathsf v}}}
 \nc{\zq}{{\zeta_q}}
 \nc\qup{{q\uparrow 1}}
 \nc{\us}{\underset}
 \nc{\tn}{{\tilde{n}}}
 \nc{\gD}{{\Delta}}
 \nc{\bi}{{\bf i}}
 \nc{\bfone}{{\bf 1}}
 \nc{\bfa}{{\bf a}}
 \nc{\bfb}{{\bf b}}
 \nc{\bfc}{{\bf c}}
 \nc{\bfd}{{\bf d}}
 \nc{\bfe}{{\bf e}}
 \nc{\bff}{{\bf f}}
 \nc{\bfg}{{\bf g}}
 \nc{\bfh}{{\bf h}}
 \nc{\bfi}{{\bf i}}
 \nc{\bfj}{{\bf j}}
 \nc{\bfn}{{\bf n}}
 \nc{\bfl}{{\bf l}}
 \nc{\bfk}{{\bf k}}
 \nc{\bfm}{{\bf m}}
 \nc{\bfo}{{\bf o}}
 \nc{\bfp}{{\bf p}}
 \nc{\bfq}{{\bf q}}
 \nc{\bfr}{{\bf r}}
 \nc{\tbfs}{{\tilde{\bf s}}}
 \nc{\bfs}{{\bf s}}
 \nc{\hbfs}{{\hat{\bf s}}}
 \nc{\hs}{{\hat{s}}}
 \nc{\ts}{\tilde{s}}
 \nc{\bft}{{\bf t}}
 \nc{\bfu}{{\bf u}}
 \nc{\bfv}{{\bf v}}
 \nc{\bfw}{{\bf w}}
 \nc{\bfx}{{\bf x}}
 \nc{\bfy}{{\bf y}}
 \nc{\bfz}{{\bf z}}
 \nc{\bfB}{{\bf B}}
 \nc{\bfP}{{\bf P}}
 \nc{\bfQ}{{\bf Q}}
 \nc{\bfY}{{\bf Y}}
 \nc{\bfgb}{{\boldsymbol \gb}}
 \nc{\bfga}{{\boldsymbol \ga}}
 \nc{\bfrho}{{\boldsymbol \rho}}
 \nc{\bfchi}{{\boldsymbol \chi}}
 \nc{\QX}{{\Q\langle \bfX\rangle}}
 \nc{\QY}{{\Q\langle \bfY\rangle}}
 \nc{\CX}{{\C\langle \bfX\rangle}}
 \nc{\CY}{{\C\langle \bfY\rangle}}
 \nc{\QXX}{{\Q\langle\!\langle \bfX\rangle\!\rangle}}
 \nc{\QYY}{{\Q\langle\!\langle \bfY\rangle\!\rangle}}
 \nc{\CXX}{{\C\langle\!\langle \bfX\rangle\!\rangle}}
 \nc{\CYY}{{\C\langle\!\langle \bfY\rangle\!\rangle}}

 \nc{\bbA}{{\mathbb A}}
 \nc{\bbB}{{\mathbb B}}
 \nc{\bbC}{{\mathbb C}}
 \nc{\bbD}{{\mathbb D}}
 \nc{\bbE}{{\mathbb E}}
 \nc{\bbF}{{\mathbb F}}
 \nc{\bbG}{{\mathbb G}}
 \nc{\bbH}{{\mathbb H}}
 \nc{\bbI}{{\mathbb I}}
 \nc{\bbJ}{{\mathbb J}}
 \nc{\bbK}{{\mathbb K}}
 \nc{\bbL}{{\mathbb L}}
 \nc{\bbM}{{\mathbb M}}
 \nc{\bbN}{{\mathbb N}}
 \nc{\bbO}{{\mathbb O}}
 \nc{\bbP}{{\mathbb P}}
 \nc{\bbQ}{{\mathbb Q}}
 \nc{\bbR}{{\mathbb R}}
 \nc{\bbS}{{\mathbb S}}
 \nc{\bbT}{{\mathbb T}}
 \nc{\bbU}{{\mathbb U}}
 \nc{\bbV}{{\mathbb V}}
 \nc{\bbW}{{\mathbb W}}
 \nc{\bbX}{{\mathbb X}}
 \nc{\bbY}{{\mathbb Y}}
 \nc{\bbZ}{{\mathbb Z}}
 \nc{\bba}{{\mathbb a}}
 \nc{\bbb}{{\mathbb b}}
 \nc{\bbc}{{\mathbb c}}
 \nc{\bbd}{{\mathbb d}}
 \nc{\bbe}{{\mathbb e}}
 \nc{\bbf}{{\mathbb f}}
 \nc{\bbg}{{\mathbb g}}
 \nc{\bbh}{{\mathbb h}}
 \nc{\bbi}{{\mathbb i}}
 \nc{\bbk}{{\mathbb k}}
 \nc{\bbl}{{\mathbb l}}
 \nc{\bbm}{{\mathbb m}}
 \nc{\bbn}{{\mathbb n}}
 \nc{\bbo}{{\mathbb o}}
 \nc{\bbp}{{\mathbb p}}
 \nc{\bbq}{{\mathbb q}}
 \nc{\bbr}{{\mathbb r}}
 \nc{\bbs}{{\mathbb s}}
 \nc{\bbt}{{\mathbb t}}
 \nc{\bbu}{{\mathbb u}}
 \nc{\bbv}{{\mathbb v}}
 \nc{\bbw}{{\mathbb w}}
 \nc{\bbx}{{\mathbb x}}
 \nc{\bby}{{\mathbb y}}
 \nc{\bbz}{{\mathbb z}}

 \nc{\calA}{{\mathcal A}}
 \nc{\calB}{{\mathcal B}}
 \nc{\calC}{{\mathcal C}}
 \nc{\calD}{{\mathcal D}}
 \nc{\calE}{{\mathcal E}}
 \nc{\calF}{{\mathcal F}}
 \nc{\calG}{{\mathcal G}}
 \nc{\calH}{{\mathcal H}}
 \nc{\calI}{{\mathcal I}}
 \nc{\calJ}{{\mathcal J}}
 \nc{\tcalI}{{\tilde{\mathcal I}}}
 \nc{\tcalJ}{{\tilde{\mathcal J}}}
 \nc{\calK}{{\mathcal K}}
 \nc{\calL}{{\mathcal L}}
 \nc{\calM}{{\mathcal M}}
 \nc{\calN}{{\mathcal N}}
 \nc{\calO}{{\mathcal O}}
 \nc{\calP}{{\mathcal P}}
 \nc{\calQ}{{\mathcal Q}}
 \nc{\calR}{{\mathcal R}}
 \nc{\calS}{{\mathcal S}}
 \nc{\calT}{{\mathcal T}}
 \nc{\calU}{{\mathcal U}}
 \nc{\calV}{{\mathcal V}}
 \nc{\calW}{{\mathcal W}}
 \nc{\calX}{{\mathcal X}}
 \nc{\calY}{{\mathcal Y}}
 \nc{\calZ}{{\mathcal Z}}
  \nc{\cala}{{\mathcal a}}
 \nc{\calb}{{\mathcal b}}
 \nc{\calc}{{\mathcal c}}
 \nc{\cald}{{\mathcal d}}
 \nc{\cale}{{\mathcal e}}
 \nc{\calf}{{\mathcal f}}
 \nc{\calg}{{\mathcal g}}
 \nc{\calh}{{\mathcal h}}
 \nc{\cali}{{\mathcal i}}
 \nc{\calj}{{\mathcal j}}
 \nc{\calk}{{\mathcal k}}
 \nc{\call}{{\mathcal l}}
 \nc{\calm}{{\mathcal m}}
 \nc{\caln}{{\mathcal n}}
 \nc{\calo}{{\mathcal o}}
 \nc{\calp}{{\mathsf p}}
 \nc{\calq}{{\mathcal q}}
 \nc{\calr}{{\mathcal r}}
 \nc{\cals}{{\mathcal s}}
 \nc{\calt}{{\mathcal t}}
 \nc{\calu}{{\mathcal u}}
 \nc{\calv}{{\mathcal v}}
 \nc{\calw}{{\mathcal w}}
 \nc{\calx}{{\mathcal x}}
 \nc{\caly}{{\mathcal y}}
 \nc{\calz}{{\mathcal z}}

 \nc{\frakA}{{\mathfrak A}}
 \nc{\frakB}{{\mathfrak B}}
 \nc{\frakC}{{\mathfrak C}}
 \nc{\frakD}{{\mathfrak D}}
 \nc{\frakE}{{\mathfrak E}}
 \nc{\frakF}{{\mathfrak F}}
 \nc{\frakG}{{\mathfrak G}}
 \nc{\frakH}{{\mathfrak H}}
 \nc{\frakI}{{\mathfrak I}}
 \nc{\frakJ}{{\mathfrak J}}
 \nc{\frakK}{{\mathfrak K}}
 \nc{\frakL}{{\mathfrak L}}
 \nc{\frakM}{{\mathfrak M}}
 \nc{\frakN}{{\mathfrak N}}
 \nc{\frakO}{{\mathfrak O}}
 \nc{\frakP}{{\mathfrak P}}
 \nc{\frakQ}{{\mathfrak Q}}
 \nc{\frakR}{{\mathfrak R}}
 \nc{\frakS}{{\mathfrak S}}
 \nc{\frakT}{{\mathfrak T}}
 \nc{\frakU}{{\mathfrak U}}
 \nc{\frakV}{{\mathfrak V}}
 \nc{\frakW}{{\mathfrak W}}
 \nc{\frakX}{{\mathfrak X}}
 \nc{\frakY}{{\mathfrak Y}}
 \nc{\frakZ}{{\mathfrak Z}}
 \nc{\fraka}{{\mathfrak a}}
 \nc{\frakb}{{\mathfrak b}}
 \nc{\frakc}{{\mathfrak c}}
 \nc{\frakd}{{\mathfrak d}}
 \nc{\frake}{{\mathfrak e}}
 \nc{\frakf}{{\mathfrak f}}
 \nc{\frakg}{{\mathfrak g}}
 \nc{\frakh}{{\mathfrak h}}
 \nc{\fraki}{{\mathfrak i}}
 \nc{\frakj}{{\mathfrak j}}
 \nc{\frakk}{{\mathfrak k}}
 \nc{\frakl}{{\mathfrak l}}
 \nc{\frakm}{{\mathfrak m}}
 \nc{\frakn}{{\mathfrak n}}
 \nc{\frako}{{\mathfrak o}}
 \nc{\frakp}{{\mathfrak p}}
 \nc{\frakq}{{\mathfrak q}}
 \nc{\frakr}{{\mathfrak r}}
 \nc{\fraks}{{\mathfrak s}}
 \nc{\frakt}{{\mathfrak t}}
 \nc{\fraku}{{\mathfrak u}}
 \nc{\frakv}{{\mathfrak v}}
 \nc{\frakw}{{\mathfrak w}}
 \nc{\frakx}{{\mathfrak x}}
 \nc{\fraky}{{\mathfrak y}}
 \nc{\frakz}{{\mathfrak z}}

 \nc{\sha}{{\mbox{\cyr x}}}
\nc{\so}{{{\mathfrak{so}}(5)}}
\nc{\slfour}{{{\mathfrak{sl}}(4)}}
 \nc{\sld}{{{\mathfrak{sl}}(d+1)}}
 \nc{\slr}{{{\mathfrak{sl}}(r+1)}}
 \nc{\slrr}{{{\mathfrak{sl}}(r+2)}}

 \nc{\uds}{{\underline{s}}}
\nc{\va}{{\vec a}}
\nc{\vb}{{\vec b}}
\nc{\vc}{{\vec c}}
\nc{\vdta}{{\vec \delta}}
\nc{\ve}{{\vec e}}
\nc{\vm}{{\vec m}}
\nc{\vp}{{\vec p}}
\nc{\vn}{{\vec n}}
\nc{\vmu}{{\vec \mu}}
\nc{\vr}{{\vec r}}
\nc{\vs}{{\vec s}}
\nc{\vt}{{\vec t}}
\nc{\vu}{{\vec u}}
\nc{\vx}{{\vec x}}
\nc{\vC}{{\vec C}}
\nc{\vv}{{\bf v}}

\begin{document}

\title[Witten multiple zeta function]
{Alternating Euler sums and special values of
Witten multiple zeta function attached to $\so$}
\author{Jianqiang Zhao}



\maketitle
\begin{center}
Department of Mathematics, Eckerd College, St. Petersburg, FL 33711, USA\\
Max-Planck Institut f\"ur Mathematik, Vivatsgasse 7, 53111 Bonn, Germany\\
Email: zhaoj@eckerd.edu
\end{center}

\vskip0.6cm

\noindent{\small {\bf Abstract.}
In this note we shall study the Witten multiple zeta function
associated to the Lie algebra $\so$ defined by Matsumoto. Our
main result shows that its special values at nonnegative integers
are always expressible by alternating Euler sums. More precisely,
every such special value of weight $w\ge 3$ is a finite rational linear combination
of alternating Euler sums of weight $w$ and depth at most two,
except when the only nonzero argument is one
of the two last variables in which case $\gz(w-1)$ is needed.

\vskip0.6cm

\section{Introduction}
The Witten multiple zeta function associated
to the Lie algebra $\so$ is defined by Matsumoto
as follows:
\begin{equation}\label{equ:gzsoDef}
\gz_\so(s_1,\dots,s_4)=\sum_{m,n=1}^\infty
\frac{1}{m^{s_1}n^{s_2}(m+n)^{s_3}(m+2n)^{s_4}},
\end{equation}
which converges whenever $\Re(s_1+s_3+s_4)>1$,
$\Re(s_2+s_3+s_4)>1$ and $\Re(s_1+s_2+s_3+s_4)>2$.
We call $s_1+s_2+s_3+s_4$ the \emph{weight}.
Essouabri \cite{Es} and Matsumoto \cite{Ma} have defined more general
multiple zeta functions and studied their analytic continuations.
However, the function in \eqref{equ:gzsoDef} itself already generalizes
both the zeta function $\gz_\so(s,s,s,s)$ suggested by
Zagier \cite[\S7]{Zag} after Witten \cite{W}
and the Mordell-Tornheim double zeta function \cite{Mord,Torn}
(see \S\ref{sec:MT}). Zagier and Garoufalidis independently
showed that for every positive integer $m$ there is some $c(m)\in\Q$
such that
\begin{equation}\label{equ:witten}
\gz_\so(2m,2m,2m,2m)=c(m)\cdot\pi^{8m}.
\end{equation}
Special values like these are the main objects of study in this note.

In \cite{Ts} Tsumura considered the special values of
of \eqref{equ:gzsoDef} at nonnegative integers.
In particular, when the weight is an odd number
he showed that the special values of \eqref{equ:gzsoDef}
are $\Q$-linear combinations of products of Riemann zeta values
at positive integers, with slightly stronger restrictions on
the arguments than just to guarantee convergence.
Since the function $\gz_\so(s_1,\dots,s_4)$
has depth two this type of results is commonly
referred to as a ``parity'' relation. For example,
the (Euler-Zagier) multiple zeta value (MZV for short)
at positive integers
\begin{equation}\label{equ:gzDef}
  \gz(s_1,\dots,s_d):=\sum_{m_1> \dots>m_d\ge 1}
 m_1^{-s_1} m_2^{-s_2}\cdots  m_d^{-s_d}
\end{equation}
has the well-known property that
if the weight ($>2$) and the depth have different parities
then it can be written as a $\Q$-linear combination
of products of MZVs of lower depths (see \cite{IKZ,Tsu1}).
In general it is expected that when the weight is even (and
large enough) we do not always have such relations.

In this note we will investigate all the convergent special
values of $\gz_\so(s_1,\dots,s_4)$ at nonnegative integers without
any parity restriction on the weight. It turns out that they are closely
related to the \emph{alternating Euler sums} (see \S\ref{sec:Esum}).
Our main result  is
\begin{thm} \label{thm:main}
Let $s_1,\dots,s_4$ be nonnegative integers such that
$s_1+s_3+s_4>1$, $s_2+s_3+s_4>1$ and $w:=s_1+s_2+s_3+s_4>2$.
Then $\gz_\so(s_1,\dots,s_4)$
can be expressed as a finite $\Q$-linear combination of alternating
Euler sums of weight $w$ and depths at most two, except
when $s_1=s_2=s_3=0$ or $s_1=s_2=s_4=0$ in which cases $\gz(w-1)$ is needed.
\end{thm}

As a final remark we point out that
the alternating Euler sums or more generally,
special values of multiple polylogarithms at roots of unity \cite{Zocta},
can be used to study many other types of multiple zeta functions
such as those appearing in the recent work of Komori, Matsumoto and Tsumura
\cite{KMT1,KMT2,KMT3}. This will be carried out in detail in another work.

The author would like to thank Max-Planck-Institut
f\"ur Mathematik for providing financial support
during his sabbatical leave when this work was done.

\section{Some preliminaries}
\subsection{Alternating Euler sum}\label{sec:Esum}
For positive integers $s_1,\dots,s_d\in$ we define
the \emph{alternating Euler sum} by
\begin{equation}\label{equ:z}
\zeta(s_1,\dots,s_d;x_1,\dots,x_d)
   := \sum_{m_1>\cdots>m_d\ge 1}\;
   \frac{x_1^{m_1}\cdots x_d^{m_d}}{m_1^{s_1}\cdots m_d^{s_d}}, \quad  (s_1,x_1)\ne (1,1),
\end{equation}
where $x_j=\pm 1$ for all $1\le j\le l$. We call
$s_1+\dots+s_d$ the \emph{weight} and $d$ the \emph{depth}.
To save space, if $x_j=-1$ then
$\overline s_j$ will be used. For example, we have
$\zeta(\overline1)=\zeta(1;-1)=-\ln 2$ and the striking
identity \cite{Zesum} that for every positive integer $n$
\begin{equation*}
 \zeta(\{3\}^n)=8^n\zeta(\{\bar{2},1\}^n),
\end{equation*}
where $\{S\}^n$ means the string $S$ repeats $n$ times.
Identities like these which are derived by (regularized) double
shuffle relations
will be crucial to simplify our computations in the last section.

\subsection{Mordell-Tornheim zeta functions}\label{sec:MT}
They are defined by (see \cite{Mord,Torn})
\begin{equation}\label{equ:MT}
\gz_\MT(s_1,\dots,s_d;s)=\sum_{m_1,\dots,m_d=1}^\infty
 m_1^{-s_1} m_2^{-s_2}\cdots  m_d^{-s_d}(m_1+\cdots +m_d)^{-s}.
\end{equation}
Recently Zhou and Bradley have shown \cite[Thm.~4]{ZB}
that \eqref{equ:MT} converges absolutely
if  $\Re(s)+\sum_{j=1}^\ell \Re(s_{i_j})>\ell$
for each nonempty subset $\{i_1,\dots,i_\ell\}$
of $\{1,2,\dots,d\}$.
We can use integral test and the  well-known formula
$$\sum_{m=1}^n m^t=\frac{1}{t+1}\Big(B_{t+1} (n+1) -B_{t+1}(0)\Big),$$
where $B_{t+1}(x)$ is the Bernoulli polynomial,
to extend their proof to the following
necessary and sufficient conditions for convergence
when all arguments are integers.
\begin{prop}\label{prop:convMT}
Let $s_1,\dots,s_d$ and $s$ be arbitrary integers. Then
the Mordell-Tornheim zeta function $\gz_\MT(s_1,\dots,s_d;s)$
converges if and only if
\begin{equation*}
s+\sum_{j=1}^\ell s_{i_j}>\ell
\end{equation*}
for each nonempty subset $\{i_1,\dots,i_\ell\}$
of $\{1,2,\dots,d\}$.
\end{prop}
The main result of \cite{ZB} is the following
\begin{prop} \label{prop:ZBred} \emph{(\cite[Thm.~5]{ZB})}
Let $s_1,\dots,s_d$ and $s$ be nonnegative integers. If at most
one of them is equal to $0$ then the Mordell-Tornheim zeta value
$\gz_\MT(s_1,\dots,s_d;s)$ can be expressed as a $\Q$-linear
combination of MZVs of the same weight and depth.
\end{prop}
In this note we will only need this proposition when the depth is two.

\subsection{Convergence domain of $\gz_\so(s_1,\dots,s_4)$}
In the following proposition we only consider
integer arguments although it is not hard to extend
it to the complex variable situation. The result can be
derived from the concrete singularity set given in \cite{KMT2}
but the following proof is more straight-forward.
\begin{prop}
\label{prop:conv}
Let $s_1,\dots,s_4$ be nonnegative integers. Then
$$\gz_\so(s_1,\dots,s_4)=\sum_{m,n=1}^\infty
\frac{1}{m^{s_1}n^{s_2}(m+n)^{s_3}(m+2n)^{s_4}}$$
converges if and only if
\begin{equation}\label{equ:domainOfConv}
s_1+s_3+s_4>1,
s_2+s_3+s_4>1, \text{ and }s_1+s_2+s_3+s_4>2.
\end{equation}
\end{prop}
\begin{proof}
First we observe that for all $m,n>0$
$$m+n < m+2n< 2(m+n).$$
Hence
\begin{multline*}
  \frac1{2^{s_4}}\gz_\MT(s_1,s_2;s_3+s_4)=\sum_{m,n=1}^\infty
 \frac{1}{m^{s_1}n^{s_2}(m+n)^{s_3}(2m+2n)^{s_4}}\\
 \le \sum_{m,n=1}^\infty
\frac{1}{m^{s_1}n^{s_2}(m+n)^{s_3}(m+2n)^{s_4}}
\le \sum_{m,n=1}^\infty
 \frac{1}{m^{s_1}n^{s_2}(m+n)^{s_3+s_4}}=\gz_\MT(s_1,s_2;s_3+s_4).
\end{multline*}
The proposition now follows immediately from the
the convergence criterion of the Mordell-Tornheim
double zeta function in Prop.~\ref{prop:convMT}.
\end{proof}

\subsection{A combinatorial lemma}
The following lemma will be used heavily throughout the
proof of Theorem~\ref{thm:main}.
\begin{lem} \label{lem:combLem}\emph{ (\cite[Lemma 1]{ZB})}
Let $r$ and $n_1,\dots,n_r$ be positive integers, and let
$x_1,\dots,x_r$ be non-zero real number such that $x_1+\dots+x_r\ne 0$.
Then
$$\prod_{j=1}^r \frac{1}{x_j^{n_j}} =
\sum_{j=1}^r\Bigg(\prod_{\substack{k=1\\ k\ne j}}^r
\sum_{a_k=0}^{n_k-1}\Bigg) \frac{M_j}{x^{n_j+A_j}}
    \prod_{\substack{k=1\\ k\ne j}}^r \frac{1}{x_k^{n_k-a_k}},$$
where the multi-nomial coefficient
$$M_j=\frac{(n_j+A_j-1)!}{(n_j-1)!} \prod_{\substack{k=1\\ k\ne j}}^r \frac{1}{a_k!} \qquad \text{and}\qquad A_j=\sum_{\substack{k=1\\ k\ne j}}^r a_k.$$
The notation $\displaystyle \prod_{\substack{k=1\\ k\ne j}}^r
\sum_{a_k=0}^{n_k-1}$ means the multiple sum
$\displaystyle \sum_{a_1=0}^{n_1-1}\dots\sum_{a_{j-1}=0}^{n_{j-1}-1}
\sum_{a_{j+1}=0}^{n_{j+1}-1}\dots\sum_{a_r=0}^{n_r-1}.$
\end{lem}

\section{Proof of Theorem \ref{thm:main}}
We now use a series of reductions to prove the theorem.

\medskip
\noindent
\underline{Case (i)}.
If $s_4=0$ then we just get a Mordell-Tornheim double zeta
value so the theorem is mostly handled by Prop.~\ref{prop:ZBred}
except for the case $s_1=s_2=0$. Then assuming $s\ge 3$ we have
\begin{equation}\label{equ:zetas0}
\gz_\so(0,0,s,0)=\gz(s,0)=\sum_{m>n\ge 1} \frac{1}{m^{s}}=
\sum_{m=1}^{\infty} \frac{m-1}{m^{s}}=\gz(s-1)-\gz(s).
\end{equation}
Thus Theorem \ref{thm:main} is true in this case.
This is the first one of the two exceptional cases in which we
need the Riemann zeta value with the weight lowered by one.

We now assume $s_4>0$ in the rest of the proof.
If $s_3=0$ and $s_2>0$ (resp.\ $s_2=0$ and $s_1>0$, resp. $s_1=s_2=0$)
then one can go directly to Case (iii.a) (resp.\ Case (iii.b),
resp.\ Case (iii.c)) below. Otherwise we must
have $s_2,s_3,s_4>0$ which is Case (ii) next.

\medskip
\noindent
\underline{Case (ii)}. Assume $s_2,s_3,s_4>0$.
In Lemma~\ref{lem:combLem}
taking  $x_1=n$ and $x_2=m+n$  we get:
\begin{align}
 \gz_\so(s_1,\dots,s_4)
=&\sum_{a_2=0}^{s_2-1}
{s_3+a_2-1\choose a_2}
\gz_\so(s_1,s_2-a_2,0,s_4+s_3+a_2)\label{case2line1}\\
+&\sum_{a_3=0}^{s_3-1}
{s_2+a_3-1\choose a_3}
\gz_\so(s_1,0,s_3-a_3,s_4+s_2+a_3).\label{case2line2}
\end{align}
We recommend the interested reader to check the convergence
of the above values by \eqref{equ:domainOfConv}.
The rule of thumb is as follows: if we apply Lemma~\ref{lem:combLem}
with each $x_j$ a positive combination of indices then
the convergence is automatically guaranteed.
In each of the following steps we often omit this
convergence checking since it is straight-forward
in most cases. The only exception is \eqref{equ:limit1}
which in fact poses the most difficulty.

Note that the weight is kept unchanged in the above so
we are led to the following three cases:
\begin{align*}
    \text{(iii.a)}.&  \ s_3=0,s_2, s_4>0  \text{ from \eqref{case2line1} since }s_2-a_2>0,\\
    \text{(iii.b)}.&  \ s_2=0,s_1, s_4>0 \text{ from \eqref{case2line2} if we started with $s_1>0$},\\
    \text{(iii.c)}.&  \ s_1=s_2=0, s_4>0 \text{ from \eqref{case2line2} if we started with $s_1=0$}.
\end{align*}

\medskip
\noindent
\underline{Case (iii.a)}. Suppose $s_3=0$, $s_2>0$ and $s_4>0$.
With $x_1=m$, $x_2=m+2n$ Lemma~\ref{lem:combLem} yields
\begin{equation}\label{equ:s2=0}
  \gz_\so(s_1,s_2,0,s_4)
=  2^{s_2}  \sum_{m,n=1}^\infty
\frac{1}{m^{s_1}(2n)^{s_2}(m+2n)^{s_4}}
=   2^{s_2-1} \sum_{m,n=1}^\infty
\frac{1+(-1)^n}{m^{s_1}n^{s_2}(m+n)^{s_4}}.
\end{equation}
Breaking this into two parts and applying the Lemma with
$x_1=m$, $x_2=n$ to the second part we have
\begin{align*}
\eqref{equ:s2=0}= &2^{s_2-1}\Bigg\{\gz_\MT(s_1,s_2;s_4)
+   \sum_{a_1=0}^{s_1-1} {s_2+a_1-1\choose a_1}
\sum_{m,n=1}^\infty \frac{ (-1)^n}{m^{s_1-a_1}(m+n)^{s_4+s_2+a_1}} \\
&\phantom{2^{s_2-1} \gz_\MT(s_1,s_2;s_4)}
+\sum_{a_2=0}^{s_2-1} {s_1+a_2-1\choose a_2}
\sum_{m,n=1}^\infty \frac{ (-1)^n}{n^{s_2-a_2}(m+n)^{s_4+s_1+a_2}}\Bigg\} \\
= &2^{s_2-1}\Bigg\{ \gz_\MT(s_1,s_2;s_4)
+   \sum_{a_1=0}^{s_1-1} {s_2+a_1-1\choose a_1}
\gz(\ol{s_4+s_2+a_1}, \ol{s_1-a_1})\\
&\phantom{2^{s_2-1} \gz_\MT(s_1,s_2;s_4)}
+\sum_{a_2=0}^{s_2-1} {s_1+a_2-1\choose a_2}
 \gz(s_4+s_1+a_2,\ol{s_2-a_2})\Bigg\} .
\end{align*}
Observe that the last component of every alternating Euler sum
(or double zeta value) in the above sums is positive and
its weight is unchanged. So Theorem~\ref{thm:main} holds in
this case.

\medskip
\noindent
\underline{Case (iii.b)}. Suppose $s_2=0,$ $s_1>0$ and $s_4>0$. Applying
Lemma \ref{lem:combLem} with $x_1=m$, $x_2=m+2n$ to
$$ \gz_\so(s_1,0,s_3,s_4)
=   \sum_{m,n=1}^\infty
\frac{1}{m^{s_1}(m+n)^{s_3}(m+2n)^{s_4}},
$$
we get
\begin{align}
 \gz_\so(s_1,0,s_3,s_4)
=&\sum_{a_1=0}^{s_1-1} {s_4+a_1-1\choose a_1}
 \frac1{2^{s_4+a_1}} \gz(s_3+s_4+a_1,s_1-a_1) \label{equ:case3.1line1}\\
+&\sum_{a_4=0}^{s_4-1} {s_1+a_4-1\choose a_4}
         \frac1{2^{s_1+a_4}} \gz_\so(0,0,s_3+s_1+a_4,s_4-a_4).\label{equ:s3s4left}
\end{align}
Note that all the double zeta values in \eqref{equ:case3.1line1} have the
same weight as the one we start with. We remind the reader that
to determine the weight of a MZV it's
not enough just to add all the components to see that weight
does not change. We also need to
check that every component is positive.
In particular the last component $s_1-a_1>0$ in \eqref{equ:case3.1line1}.
So we are reduced to the case (iii.c): $s_1=s_2=0$, and because of
the convergence restriction that the sum of all components
is at least 3 we may also assume that $s_3+s_4\ge 3$ holds in case (iii.c).
Further, since we assume $s_1>0$ and $s_4-a_4>0$ in \eqref{equ:s3s4left}
we are in fact reduced to the subcase (iii.c.2) of (iii.c)
where we can assume $s_3, s_4>0$.

\medskip
\noindent
\underline{Case (iii.c)}. Suppose $s_1=s_2=0$ and $s_4>0$. We divide the case further
into two subcases: (iii.c.1) $s_1=s_2=s_3=0$ and $s_4>0$,
and (iii.c.2) $s_1=s_2=0,s_3,s_4>0$.

\medskip
\noindent
\underline{Case (iii.c.1)}. In this case setting $s_4=s\ge 3$ (by convergence restraint) we get
\begin{equation*}
  \gz_\so(0,0,0,s)=\sum_{m,n=1}^\infty\frac{1}{(m+2n)^{s}}
 =\frac12\sum_{m,n=1}^\infty \frac{1+(-1)^n}{(m+n)^{s}}
 = \frac12\Big\{\gz(s,0)+\sum_{k>m\ge 1} \frac{ (-1)^{m+k} }{k^{s}}\Big\}.
\end{equation*}
Now the sum over $m$ is 0 unless $k$ is even so by \eqref{equ:zetas0} we have
\begin{equation*}
\gz_\so(0,0,0,s)=\frac12\Big\{\gz(s-1)-\gz(s)+\sum_{2k \ge 1} \frac{-1}{(2k)^{s}}\Big\}
=\frac12\Big\{\gz(s-1)-\gz(s)-2^{-s}\gz(s)\Big\}.
\end{equation*}
Thus Theorem \ref{thm:main} holds in this case.
This is the second exceptional case when we
need the Riemann zeta value with the weight decreased by one.

\medskip
\noindent
\underline{Case (iii.c.2)}.
Suppose $s_3=r>0, s_4=t>0$ and $r+t\ge 3$ (by convergence restraint).
Taking $x_1=-m-n$, $x_2=m+2n$ in Lemma~\ref{lem:combLem}
we get
\begin{align}
&  \gz_\so(0,0,r,t)= \sum_{m,n=1}^\infty
\frac{(-1)^{r}}{(-m-n)^{r}(m+2n)^{t}} \notag\\
= & \sum_{a=0}^{r-2} {t+a-1\choose a}
\sum_{m,n=1}^\infty \frac{(-1)^{a}}{n^{t+a}(m+n)^{r-a}}
+ \sum_{a=0}^{t-2} {r+a-1\choose a}
\sum_{m,n=1}^\infty \frac{(-1)^{r}}{n^{r+a}(m+2n)^{t-a}}  \notag\\
&\hskip2cm+(-1)^{r} {t+r-2\choose r-1}\sum_{m,n=1}^\infty
\Bigg( \frac{1}{n^{r+t-1}(m+2n)}- \frac{1}{n^{r+t-1}(m+n)}\Bigg)\label{equ:limit1}
\end{align}
The inner infinite sum of the first sum is exactly $(-1)^{a}\gz(r-a,t+a)$ while
the second sum can be dealt with by the method
similar to \eqref{equ:s2=0}. For any positive integer $u,v$ with $v>1$
we have
$$\sum_{m,n=1}^\infty \frac{1}{n^u(m+2n)^v}
 =2^{u-1}\sum_{m,n=1}^\infty \frac{1+(-1)^n}{n^u(m+n)^v}
=2^{u-1}\Big\{\gz(v,u)+\gz(v,\bar{u})\Big\}.$$
Notice that all the alternating Euler sums have the
same weight (and the second components are all positive since $r,t>0$).
So now we need to consider the sum in \eqref{equ:limit1}.
Assume $s=r+t-1>1$ and let
\begin{align*}
 S^{(N)}:=\sum_{m,n=1}^N
\Bigg( \frac{1}{n^s (m+2n)}- \frac{1}{n^s(m+n)}\Bigg)
=& \sum_{n=1}^N
\frac{1}{n^s} \Bigg( \sum_{m=1+2n}^{N+2n}-\sum_{m=1+n}^{N+n} \Bigg) \frac{1}{m} \\
=&\sum_{n=1}^N
\frac{1}{n^s} \Bigg( \sum_{m=N+n+1}^{N+2n}+\sum_{m=1}^{n}-\sum_{m=1}^{2n} \Bigg) \frac{1}{m}.
\end{align*}
Noticing that $s>1$ and therefore
\begin{equation*}
 \sum_{n=1}^N
\frac{1}{n^s}  \sum_{m=N+n+1}^{N+2n} \frac{1}{m}
<\sum_{n=1}^N \frac{1}{n^{s-1} N}\ll \log N/N\to 0 \quad\text{ as } N\to \infty,
\end{equation*}
we quickly see that
\begin{align*}
 \lim_{N\to\infty}S^{(N)}
 =&\sum_{n\ge m\ge 1}  \frac{1}{n^sm } - \sum_{2n\ge m\ge 1} \frac{1}{n^sm } \\
 =&\gz(s+1)+\gz(s,1) - 2^{s-1}\sum_{n\ge m\ge 1} \frac{1+(-1)^n}{n^sm } \\
 =&(1 - 2^{s-1})\Big(\gz(s+1)+\gz(s,1)\Big) -2^{s-1}\Big(\gz(\bar{s},1)+\gz(\ol{s+1})\Big),
\end{align*}
where all the Euler sums have the same weight.
This concludes the proof of Theorem~\ref{thm:main}. \hfill $\Box$.

\begin{rem}\label{rem:parity}
Notice that Theorem~\ref{thm:main} does not imply Tsumura's result
concerning odd weight values since the parity relations do not hold
in general for alternating Euler sums. For example, we know
the $\Q$-linear space generated by the Riemann zeta values of
weight three has dimension one since $\gz(3)=\gz(2,1)$.
Broadhurst conjectures that the dimension $F_n$ of the space generated by
weight $n$ alternating Euler sums is a Fibonacci number:
$F_1=1$, $F_2=2$, $F_3=3$, $F_4=5$, and so on.
It is easy to verify \cite{Zbasis} that the weight three space
is spanned by $\gz(3),$ $\gz(\bar{1},2)$ and
$\gz(\bar{1},1,1)$. Also note that the depth one subspace is generated
by $\gz(3)$ since $\gz(\bar3)=-\frac34\gz(3)$. Therefore the
alternating Euler sum $\gz(\bar{1},2)$ can not be reduced to depth one
according to Broadhurst conjecture despite the parity
difference between its weight and its depth.
\end{rem}

\section{Some examples and a conjecture}
In this last section we present some numerical examples
and put forward a conjecture on the space generated by
the special values of $\gz_\so$. In \cite{Ts} Tsumura
provided some evaluations of $\gz_\so(s_1,\dots,s_4)$ when the weight
is odd. By our general approach we can now compute all
the convergent values and in particular
we are able to confirm all the odd weight values of $\gz_\so$
found in \cite{Tsu1}.
In practice one may first convert our formulas to computer programs
and then compute with Maple. As a safeguard we have checked numerically
all the equations in this section with EZface \cite{EZface}.
In what follows we will only consider the regular cases, i.e., not
the two exceptional cases in each weight.

We first list all the regular weight three values below:
$$\begin{array}{rl}
\ &\gz_\so(1,0,0,2)=\frac32\gz(\bar1,2)+\frac1{16}\gz(3),
\quad  \gz_\so(0,0,1,2)=-\frac12\gz(3)+3\gz(\bar1,2),\\
\ &\gz_\so(0,1,0,2)=\frac54\gz(3)-3\gz(\bar1,2),
\quad \ \, \gz_\so(1,0,2,0)=\gz_\so(0,1,2,0)=\gz(3),\\
\ &\gz_\so(0,1,1,1)=\frac34\gz(3),
\quad \gz_\so(1,1,1,0)=2\gz(3),
\quad  \gz_\so(1,0,1,1)=\frac58\gz(3),\\
\ &\gz_\so(0,0,2,1)=\frac14\gz(3),
\quad\gz_\so(1,1,0,1)=\frac{11}{8}\gz(3).
\end{array}$$
Note that the first three values do not satisfy the conditions
of Tsumura's Theorem in \cite{Tsu1} so there is no contradiction
even if they can not be rational multiples of $\gz(3)$ by
Broadhurst conjecture.
On the other hand, $\gz_\so(0,1,1,1)=\frac34\gz(3)$ does not
satisfy the conditions either but it is reduced to depth one.
Therefore it would be interesting to find out the exact conditions
on the reducibility of special values of $\gz_\so$ to Riemann
zeta values.

\begin{rem}
The smallest weight of $\gz_\so(s_1,s_2,s_3,s_4)$ is three because of
the convergence restraint \eqref{equ:domainOfConv}. We have listed
all the possible regular weight three values
in the above and find that
they do not span the whole weight three space of
alternating Euler sums because the whole space has dimension three according to
Broadhurst conjecture (see Remark~\ref{rem:parity}). In fact, this was
expected because $\gz_\so(s_1,s_2,s_3,s_4)$ only has depth two
while all the depth one and two alternating Euler sums only generate a
proper subspace of dimension two over $\Q$, and in fact, this subspace
can be generated by $\gz(3)$ and $\gz(\bar1,2)$. We believe this phenomenon
happens in general.
\end{rem}
\begin{conj}
Let $w$ be a positive integer $\ge 3$.
Let $V_w$ be the $\Q$-vector space spanned
by all the weight $w$ special values of $\gz_\so(s_1,s_2,s_3,s_4)$
where $s_1,s_2,s_3$ and $s_4$ are nonnegative integers
such that \eqref{equ:domainOfConv} are satisfied,
$s_1+s_2+s_3>0$ and $s_1+s_2+s_4>0$. Then $V_w$ coincides
with the $\Q$-vector space spanned by all the weight $w$
alternating Euler sums of depth at most two.
\end{conj}
We have verified this conjecture
for all the weights up to weight five.

We now list all the 25 regular weight four values below:
$$\begin{array}{rl}
&\gz_\so(0,0,2,2)=\frac38\gz(4)-4\gz(\bar{3},1),
\quad  \gz_\so(2,1,0,1)=\frac78\gz(4)-\gz(\bar{3},1), \\
\ &\gz_\so(0,1,2,1)=\frac12\gz(4)-4\gz(\bar{3},1),
\quad \gz_\so(0,2,0,2)=\frac18\gz(4)+4\gz(\bar{3},1), \\
\ &\gz_\so(0,2,1,1)=\frac14\gz(4)+4\gz(\bar{3},1),
\quad \gz_\so(1,2,0,1)=\frac34\gz(4)+2\gz(\bar{3},1), \\
\ &\gz_\so(1,0,1,2)=\frac3{16}\gz(4)-\gz(\bar{3},1),
\quad \gz_\so(2,0,1,1)=\frac38\gz(4)+\gz(\bar{3},1), \\
\ &\gz_\so(1,1,0,2)=\frac5{16}\gz(4)-\gz(\bar{3},1),
\quad \gz_\so(1,1,1,1)=\frac12\gz(4)-2\gz(\bar{3},1), \\
\ &\gz_\so(0,1,0,3)=\frac{17}{24}\gz(4)-\frac73\gz(\bar{1},3),
\quad \gz_\so(0,0,3,1)=-\frac14\gz(4)+4\gz(\bar{3},1),\\
\ & \gz_\so(0,0,1,3)=-\frac7{12}\gz(4)+\frac73\gz(\bar{1},3),
\quad \gz_\so(1,0,2,1)=2\gz(\bar{3},1),\\
\ &\gz_\so(1,0,3,0)=\gz_\so(0,1,3,0)=\frac14\gz(4),
\quad \gz_\so(0,1,1,2)=\frac18\gz(4), \\
\ &\gz_\so(2,0,2,0)=\gz_\so(0,2,2,0)=\frac34\gz(4),
\quad \gz_\so(2,2,0,0)=\frac52\gz(4) , \\
\ &\gz_\so(2,1,1,0)=\gz_\so(1,2,1,0)=\frac54\gz(4),
\quad \gz_\so(2,0,0,2)=\frac9{32}\gz(4), \\
 \ &\gz_\so(1,0,0,3)=-\frac{19}{96}\gz(4)+\frac76\gz(\bar{1},3)-\frac12\gz(\bar{3},1),
\quad \gz_\so(1,1,2,0)=\frac12\gz(4).
\end{array}$$
There are 46 regular weight five values and 74 regular weight six values.
The following are some interesting weight six ones:
$$\begin{array}{rl}
&\gz_\so(0,2,2,2)=\frac{1}{105}\gz(2)^3=\frac{1}{22680}\pi^6=.04238929428\dots\\
&\gz_\so(2,0,2,2)=\frac{1}{210}\gz(2)^3+\frac{3}{8}\gz(3,3)-\frac{2}{3}\gz(\bar{3},3)=.03772580207\dots\\
&\gz_\so(2,2,0,2)=\frac{4}{105}\gz(2)^3-\frac{3}{8}\gz(3,3)+\frac{2}{3}\gz(\bar{3},3)=.15302602205\dots\\
&\gz_\so(2,2,2,0)=\gz_\MT(2,2;2)=\frac{8}{105}\gz(2)^3=\frac{1}{2835}\pi^6 =.3391143543  \dots\\
&\gz_\so(1,2,2,1)=\frac{1}{30}\gz(2)^3-\frac{3}{4}\gz(3,3)+\frac{4}{3}\gz(\bar{3},3)=.1153002199792 \dots\\
&\gz_\so(2,1,1,2)=\frac{1}{60}\gz(2)^3=\frac{1}{12960}\pi^6=.07418126500\dots\\
&\gz_\so(1,1,2,2)=\frac{1}{84}\gz(2)^3-\frac{3}{8} \gz(3,3)+\frac{2}{3}\gz(\bar{3},3)=.0364554628649\dots \\
&\gz_\so(1,2,1,2)=\frac{3}{140}\gz(2)^3-\frac{3}{8}\gz(3,3)+\frac{2}{3}\gz(\bar{3},3)=.0788447571142\dots\\
&\gz_\so(2,1,2,1)=\frac{3}{140}\gz(2)^3+\frac{3}{8}\gz(3,3)-\frac{2}{3}\gz(\bar{3},3)=.1119070670077\dots\\
&\gz_\so(2,2,1,1)=\frac{23}{420}\gz(2)^3-\frac{3}{8}\gz(3,3)+\frac{2}{3}\gz(\bar{3},3)=.2272072869870\dots
\end{array}$$


Finally, we return to the question related to Zagier's original
version of Witten's zeta function attached to $\so$.
We would like to know the rational coefficients in \eqref{equ:witten} so
the following four values can shed some light on this:
$$\begin{array}{rl}
&\gz_\so(2,2,2,2)=\frac{3}{700}\gz(2)^4 =\frac{2\cdot 3}{5\cdot 9!}\pi^8=.031377417381\dots\\ 
&\gz_\so(4,4,4,4)=\frac{4311}{297797500}\gz(2)^8=\frac{2^5\cdot 479}{5\cdot 17! }\pi^{16}=.0007759700\dots\\
&\gz_\so(6,6,6,6)=\frac{2490861}{45593675752625}\gz(2)^{12}
=\frac{ 2^7\cdot 5\cdot43\cdot 19309 }{3^2\cdot 7\cdot 13\cdot 23!}\pi^{24}=.00002144010\dots\\
&\gz_\so(8,8,8,8)=\frac{138835874547}{670007833199392187500}\gz(2)^{16}
=\frac{2^8\cdot 13\cdot241\cdot 64009163 }{5\cdot 17\cdot 31!} \pi^{32}=.000000595384  \dots.
\end{array}$$
For $\gz_\so(2,2,2,2)$ we can reduce our linear combination of alternating
Euler sums to the correct multiple of $\gz(2)^4$. But we can not perform
the same for the other three because there isn't any table of relations available
for alternating Euler sums of weight greater than 15 although we can check
numerically that our formulas are all correct (with the errors bounded by $10^{-100}$).
However, it turns out that by analytical methods Komori et al. \cite{KMT4} have
found a closed formula of $\gz_\so(2m,2m,2m,2m)$ which implies that
\begin{align*}
c(m)=\frac{2^{8m-3}}{(8m)!}\sum_{\nu=0}^m  B_{2\nu}B_{8m-2\nu}&{8m\choose 2\nu}
\Bigg\{\sum_{\mu=0}^{2m-1} \left(\frac{2^{2\nu-1}}{2^{2m+\mu}}-(-1)^\mu\right)
{4m-\mu-2\choose 2m-1}{2m-2\nu+\mu\choose 2m-2\nu}\\
+& \sum_{\mu=0}^{2m-2\nu} \left(\frac1{2^{2m+\mu}}+(-1)^\mu\right)
{4m-2\nu-\mu-1\choose 2m-1}{2m-1+\mu\choose 2m-1}\Bigg\},
\end{align*}
where $B_{2\nu}$ are Bernoulli numbers.
So we have come around a full circle and back to the original
values we began with, but this time the rational coefficients
of the $\pi$-powers are determined precisely.


\begin{thebibliography}{99}
\bibitem{EZface}
J.~Borwein, P.~Lisonek, and P.~Irvine,\emph{ An interface for
evaluation of Euler sums}. Available online at
\url{http://oldweb.cecm.sfu.ca/cgi-bin/EZFace/zetaform.cgi}

\bibitem{Es}
D.\ Essouabri, Singularit\'es des s\'eries de Dirichlet associées
\`a des polyn\^omes de plusieurs variables et applications en th\'eorie
analytique des nombres. Ann.\ Inst.\ Fourier \textbf{47} (1997), 429--483.

\bibitem{Hoff}
M.~E.~Hoffman, \emph{Multiple harmonic series.}
Pacific J. Math.\ \textbf{152} (1992), 275--290.

\bibitem{IKZ}
K.\ Ihara, M.\ Kaneko, and D.\ Zagier, \emph{Derivation and double shuffle
relations for multiple zeta values.} Compositio Math.
\textbf{142} (2006), 307--338.

\bibitem{KMT1}
Y. Komori, K. Matsumoto and H. Tsumura,
Zeta and $L$-functions and Bernoulli polynomials of root systems,
Proc.\ Japan Acad.\ Ser. A \textbf{84} (2008), 57--62.

\bibitem{KMT2}
Y.\ Komori, K.\ Matsumoto and H.\ Tsumura,
On Witten multiple zeta-functions associated with semisimple Lie algebras II.
Preprint.

\bibitem{KMT3}
Y.\ Komori, K.\ Matsumoto and H.\ Tsumura,
On Witten multiple zeta-functions associated with semisimple Lie algebras III.
Preprint.

\bibitem{KMT4}
Y.\ Komori, K.\ Matsumoto and H.\ Tsumura,
Functional equations for zeta-functions of root systems. In preparation.

\bibitem{Ma}
K.~Matsumoto, \emph{On Mordell-Tornheim and other multiple zeta-functions}.
In: Proc.\ of the Session in analytic number theory and
Diophantine equations, ed.\ D.R.\ Heath-Brown and B.Z.\ Moroz,
Bonner Mathematische Schriften, \textbf{25} (2003), 17 pages.

\bibitem{Ma2}
K.~Matsumoto, \emph{Analytic properties of multiple zeta-functions
in several variables.}  (English summary)   Number theory,  153--173,
Dev. Math. \textbf{15}, Springer, New York, 2006.

\bibitem{Mats2}
K.~Matsumoto and H.\ Tsumura,
\emph{On Witten multiple zeta-functions associated with semisimple Lie algebras I},
Ann.\ Inst.\ Fourier \textbf{56} (2006), 1457--1504.

\bibitem{Mord}
L.J.~Mordell, \emph{On the evaluation of some multiple
series.} J.\ London Math.\ Soc.\ \textbf{33} (1958), 368--271.

\bibitem{Torn}
L.~Tornheim, \emph{Harmonic double series}. Amer.\ J.\
Math. \textbf{72} (1950), 303--314.

\bibitem{Tsu1}
H.~Tsumura, \emph{Combinatorial relations for Euler-Zagier sums.}
Acta Arith.\ \textbf{111} (2004), 27--42.

\bibitem{Ts}
H. Tsumura, On Witten's type of zeta values attached to $SO$(5),
Arch.\ Math.\ \textbf{84} (2004), 147--152.

\bibitem{W}
E.\ Witten, \emph{On quantum gauge theories in two-dimensions.}
Commun.\ Math.\ Phys.\ \textbf{141}(1) (1991), 153--209.

\bibitem{Zag} D.~Zagier, \emph{Values of zeta
function and their applications}. Proc.\ of the First
European Congress of Math. \textbf{2} (1994), 497--512.

\bibitem{Zana}
J.~Zhao,  \emph{Analytic continuation of multiple zeta
functions}. Proc.\ of Amer.\ Math.\ Soc.\ \textbf{128} (1999), 1275--1283.

\bibitem{Zesum}
J.~Zhao,  \emph{On a conjecture of Borwein, Bradley and Broadhurst.} In press:
J.\ reine angew.\ Math.

\bibitem{Zocta}
J.~Zhao, \emph{Multiple polylogarithm values
at roots of unity}. C.\ R.\ Acad.\ Sci.\ Paris, Ser. I.
\textbf{346} (2008), 1029--1032.

\bibitem{Zbasis}
J.~Zhao, Integral structures of multi-polylogs at roots of unity. In preparation.

\bibitem{ZB}
X.~Zhou and D.M.~Bradley,
On Mordell-Tornheim sums and multiple zeta values. Preprint, 2008.
\end{thebibliography}
\end{document}